\documentclass[a4paper,12pt,dvipdfmx]{amsart}
\usepackage{amsmath}
\usepackage{amssymb}
\usepackage{enumerate}
\usepackage{amscd}
\usepackage{graphics}
\usepackage{latexsym}
\usepackage{verbatim} 
\usepackage{url}

\theoremstyle{plain}
\newtheorem{thm}{{\bf Theorem}}[section]
\newtheorem{cor}[thm]{{\bf  Corollary}}
\newtheorem{prop}[thm]{{\bf Proposition}}
\newtheorem{lemma}[thm]{{\bf Lemma}}
\newtheorem{fact}[thm]{{\bf Fact}}

\newtheorem{claim}[thm]{{\bf Claim}}

\theoremstyle{definition}
\newtheorem{define}[thm]{{\bf Definition}}

\newtheorem{question}[thm]{{\bf Question}}

\newcommand{\size}[1]{\left\vert {#1} \right\vert}
\newcommand{\p}{\mathcal{P}}

\newcommand{\seq}[1]{\langle {#1} \rangle}

\newcommand{\ka}{\kappa}
\newcommand{\la}{\lambda}
\newcommand{\om}{\omega}

\newcommand{\pkl}{\mathcal{P}_\ka \lambda}

\newcommand{\bbP}{\mathbb{P}}

\newcommand{\bbR}{\mathbb{R}}

\newcommand{\calA}{\mathcal{A}}

\title[]{A note on the tightness of $G_\delta$-modifications}
\author[T. Usuba]{Toshimichi Usuba}
\address[T. Usuba]
{Faculty of Fundamental Science and Engineering,
Waseda University, 
Okubo 3-4-1, Shinjyuku, Tokyo, 169-8555 Japan}
\email{usuba@waseda.jp}

\begin{document}
\begin{abstract}
We construct a normal countably tight  $T_1$ space $X$ with $t(X_\delta) >2^\om$.
This is an answer to the question posed by Dow-Juh\'asz-Soukup-Szentmikl\'ossy-Weiss \cite{DJSSW}.
We also show that if the continuum is not so large,
then the tightness of $G_\delta$-modifications of countably tight spaces can be 
arbitrary large up to the least $\om_1$-strongly compact cardinal.
\end{abstract}

\maketitle

\section{Introduction}
For a topological space $X$, let $X_\delta$ be the $G_\delta$-modification of $X$,
that is, $X_\delta$ is the space $X$ equipped with topology generated by all $G_\delta$-subsets of $X$.

Bella and Spadaro \cite{BS} studied
the connection between the values of various cardinal functions taken
on $X$ and $X_\delta$, respectively. In their paper they posed the following question:
Is $t(X_\delta) \le 2^{t(X)}$ true for every (compact)
$T_2$ space $X$?
Recall that $t(X)$, the tightness number of $X$, is the least infinite cardinal $\ka$
such that for every $A \subseteq X$ and $p \in \overline{A}$,
there is $B \in [A]^{\le \ka}$ with $p \in \overline{B}$.
If $t(X)=\om$, $X$ is said to be \emph{countably tight}.

For this question, Dow-Juh\'asz-Soukup-Szentmikl\'ossy-Weiss \cite{DJSSW}
answered as follows:
\begin{fact}[\cite{DJSSW}]\label{fact_DJSSW}
\begin{enumerate}
\item If $X$ is a regular Lindel\"of $T_1$ space, then
$t(X_\delta)\le 2^{t(X)}$.
\item Under $V=L$,  for every cardinal $\ka$
there is a Fr\'echet-Urysohn space $X$ with $t(X_\delta) \ge \ka$.
\end{enumerate}
\end{fact}
The clause (1) of Fact \ref{fact_DJSSW} 
is a theorem of ZFC.
However (2) is a consistency result,
and they asked the following natural question:
\begin{question}
Is there a ZFC example of a countably tight Hausdorff
(or regular, or Tychonoff) space $X$ for which $t(X_\delta) > 2^{t(X)}$?
\end{question}
In this paper we give a positive answer to their question.
\begin{thm}\label{thm1.3}
There is a normal countably tight $T_1$ space $X$ such that $t(X_\delta) >2^\om$.
\end{thm}

We  also observe 
some connection between $\om_1$-strongly compact cardinals and
the tightness of $G_\delta$-modifications.
Usuba \cite{Usuba} studied 
the Lindel\"of number of $G_\delta$-modifications of compact spaces,
and proved the following equality:
\[
\text{the least $\om_1$-strongly compact}
=\sup\{L(X_\delta) \mid \text{$X$ is compact $T_2$\,}\}.
\]

Under some assumption, we  prove
similar results for the tightness of $G_\delta$-modifications.

\begin{thm}\label{thm1.4}
\begin{enumerate}
\item Suppose $\ka$ is the least $\om_1$-strongly compact cardinal.
Then for every countably tight space $X$ we have $t(X_\delta) \le \ka$.
\item Suppose there is no weakly Mahlo cardinal $ <2^\om$ (e.g., CH holds).
If there is no $\om_1$-strongly compact cardinal below a cardinal $\nu$,
then there is a normal countably tight $T_1$ space $X$ with $t(X_\delta) \ge \nu$.
%
%then for every cardinal $\nu$,
%there is a countably tight normal $T_1$ space $X$ such that $t(X_\delta)  \ge \nu$.
%\item If $\ka$ is the least $\om_1$-strongly compact cardinal,
%then for every cardinal $\nu<\ka$,
%there is a countably tight normal $T_1$ space $X$ such that $t(X_\delta)  \ge \nu$.
%\end{enumerate}
\end{enumerate}
\end{thm}

Thus, assuming that $2^\om$ is not so large,
we have the following equality:
\[
\text{the least $\om_1$-strongly compact}
=\sup\{t(X_\delta) \mid \text{$X$ is normal countably tight $T_1$\,}\}.
\]

Here we present some notations, definitions, and facts.

For a topological space $X$ and $A \subseteq X$, let $\overline{A}^\delta$ be the closure of $A$ in $X_\delta$.

For a filter $F$ over the set $S$ and a cardinal $\ka$,
let us say that $F$ is \emph{$\ka$-complete}
if for every family $\calA \subseteq F$ of size $<\ka$,
we have $\bigcap \calA \in F$.
A filter $F$ is \emph{$\ka$-incomplete} if
$F$ is not $\ka$-complete.

The concept of \emph{$\om_1$-strongly compact cardinal} is introduced by Bagaria and Magidor.
\begin{define}[Bagaria-Magidor \cite{BM1, BM2}]
An uncountable cardinal $\ka$ is \emph{$\om_1$-strongly compact}
if for every set $S$ and every $\ka$-complete filter $F$ over $S$,
the filter $F$ can be extended to an $\om_1$-complete ultrafilter over $S$.
%\begin{enumerate}
%\item Let $a$ be a countably infinite set,
%and $x, y \subseteq a$ are infinite subsets.
%Let us say that $x$ \emph{splits} $y$ if
%both $y \cap x$ and $y \setminus x$ are infinite.
%A family $\mathcal{R}$ of infinite subsets of $a$
%is \emph{unsplittable} if no single set splits all members of $\mathcal{R}$.
%The \emph{reaping  number} $\mathfrak{r}$ is 
%the minimal cardinality of unsplittable family.
\end{define}
Note that if $\ka$ is $\om_1$-strongly compact,
then every cardinal greater than $\ka$ is $\om_1$-strongly compact.

\begin{define}
\begin{enumerate}
\item For an uncountable cardinal $\ka$ and a set $A$,
let $\p_\ka A=\{x \subseteq A \mid \size{x}<\ka\}$.
\item A filter $F$ over $\p_\ka A$ is \emph{fine}
if for every $a \in A$, we have $\{x \in \p_\ka A \mid a \in x\} \in F$.
\end{enumerate}
\end{define}

\begin{fact}[Bagaria-Magidor \cite{BM1, BM2}]\label{1.6}
\begin{enumerate}
\item An uncountable cardinal $\ka$ is $\om_1$-strongly compact
if and only if
 for every cardinal $\la \ge \ka$,
there exists an $\om_1$-complete fine ultrafilter over $\pkl$.
\item If $\ka$ is the least $\om_1$-strongly compact,
then $\ka$ is a limit cardinal and there exists a measurable cardinal $\le \ka$.
\item It is possible that
the least $\om_1$-strongly compact is a singular cardinal.
\end{enumerate}
\end{fact}

Now we give the proof of (1) in Theorem \ref{thm1.4}.
The proof is essentially the same to in 
Dow-Juh\'asz-Soukup-Szentmikl\'ossy-Weiss \cite{DJSSW},
but we give it for the completeness.
\begin{prop}
Let $\ka$ be the least $\om_1$-strongly compact.
Then for every countably tight topological space $X$,
$A \subseteq X$, and $p \in \overline{A}^\delta$,
there is $B \subseteq A$ with $\size{B}<\ka$ and $p \in \overline{B}^\delta$.
Hence $t(X_\delta) \le \ka$.
\end{prop}
\begin{proof}
We may assume that $A=\la$ for some cardinal $\la \ge \ka$.
By Fact \ref{1.6},
there is an $\om_1$-complete fine ultrafilter $U$ over $\pkl$.

Suppose to the contrary that 
$p \notin \overline{B}^\delta$ for every $B \subseteq \la$ with $\size{B}<\ka$.
For $B \subseteq \la$ with $\size{B}<\ka$,
there are open neighborhoods $O^n_B$ ($n<\om$) of $p$ with $B \cap \bigcap_{n<\om} O^n_B=\emptyset$.
Since $U$ is $\om_1$-complete,
for each $\alpha<\la$,
there is $n<\om$
with
$\{B \in \pkl \mid \alpha \in B \setminus \bigcap_{i \le n} O_B^i\} \in U$.
For $n<\om$,
let $A_n$ be the set of all $\alpha<\la$ with
$\{B \in \pkl \mid \alpha \in B \setminus \bigcap_{i \le n} O_B^i\} \in U$.
We have $\la=\bigcup_{n<\om} A_n$.
On the other hand we have $p \notin \overline{A_n}$;
If $p \in \overline{A_n}$,
there is a countable $C \subseteq A_n$ with
$p \in \overline{C}$.
Since $U$ is $\om_1$-complete,
we can find $B \in \pkl$ with
$\alpha \in B \setminus \bigcap_{i \le n} O_B^i$ for every $\alpha \in C$,
this means that $p \in \bigcap_{i \le n} O_B^i$ but $C \cap \bigcap_{i\le n}O_B^i=\emptyset$, this is impossible.
Thus $p \notin \overline{A_n}$,
and this immediately implies  that $p \notin \overline{A}^\delta$.
\end{proof}

\section{Construction of the spaces}

For the sake of constructing our spaces,
we use the function spaces.
Let us recall some definitions and basic facts.
For a Tychonoff space $X$, let $C(X)$ be the 
set of all continuous functions from $X$ into the real line $\bbR$.
$C_p(X)$ is the space $C(X)$ endowed with the point-wise convergence,
that is, the topology of $C_p(X)$ is 
generated by the family
$\{ V(x_0,\dotsc, x_n, O_0,\dotsc, O_n) \mid x_0,\dotsc, x_n \in X, O_0,\dotsc, O_n \subseteq \bbR$ are open in $\bbR\}$
where 
$V(x_0,\dotsc, x_n, O_0,\dotsc, O_n)$ is the set of all
$f \in C(X)$ with $f(x_i) \in O_i$ for every $i \le n$.

\begin{fact}[Arhangel'ski\u \i-Pytkeev \cite{A,P}]\label{1.1}
Let $X$ be a Tychonoff space, and $\nu$ a cardinal.
Then $L(X^n) \le \nu$ for every $n<\om$
if and only if
$t(C_p(X)) \le \nu$.
In particular,
each finite product of $X$ is Lindel\"of if and only if $C_p(X)$ is countably tight.
\end{fact} 

\begin{prop}\label{1.2}
Let $\ka$ be an uncountable cardinal and $\la \ge \ka$ a cardinal.
Suppose there is no $\om_1$-complete fine ultrafilter over $\pkl$.
In addition we suppose that, for every countable family $\{U_n \mid n<\om\}$ of
fine ultrafilters over $\pkl$, 
there is a countable partition $\calA$ of $\pkl$
such that $A \notin U_n$ for every $A \in \calA$ and $n<\om$.
Then there is a countably tight Tychonoff space $X$ with $t(X_\delta) \ge \ka$.
\end{prop}
\begin{proof}
Identifying $\pkl$ as a discrete space,
let $\mathrm{Fine}(\pkl)$ be the closed subspace of the Stone-\v Cech compactification $\beta(\pkl)$ consisting of all
fine ultrafilters over $\pkl$.
Let $X=C_p(\mathrm{Fine}(\pkl))$.
Since $\mathrm{Fine}(\pkl)$ is compact Hausdorff,
each finite product of 
$\mathrm{Fine}(\pkl)$ is compact.
Hence $X$ is countably tight by Fact \ref{1.1}.
We shall show that $t(X_\delta) \ge \ka$.

Let $\{\calA^\alpha \mid \alpha<\nu\}$ be an enumeration of all countable partitions of $\pkl$.
For $\alpha<\nu$, 
let $S^{\calA^\alpha}=\{U \in \mathrm{Fine}(\pkl) \mid A \notin U$ for every $A \in \calA^\alpha\}$.
$S^{\calA^\alpha}$ is a closed $G_\delta$-subset of $\mathrm{Fine}(\pkl)$.
Since there is no $\om_1$-complete fine ultrafilter over $\pkl$,
the family $\{S^{\calA^\alpha} \mid \alpha<\nu\}$ is a cover of
$\mathrm{Fine}(\pkl)$.
Furthermore, by our assumption, 
for every $U_n \in \mathrm{Fine}(\pkl)$ ($n<\om$)
there is $\alpha<\nu$ such that $U_n \in S^{\calA^\alpha}$ for every $n<\om$.

We use the following fact:
\begin{fact}[Usuba \cite{Usuba}]\label{1.3}
If $\{S^{\calA^\alpha} \mid \alpha \in E\}$ is a cover of $\mathrm{Fine}(\pkl)$ for some $E \subseteq \nu$, then $\size{E} \ge \ka$.
\end{fact}
\begin{proof}[Sketch of the proof]
Take $E \subseteq \nu$ with size $<\ka$. 
Take a sufficiently large regular $\theta$,
and take $M \prec H_\theta$ containing all relevant objects such that
$\size{M}<\ka$ and $E \subseteq M$.
For each $\alpha \in E$, there is a unique $A_\alpha \in \calA^\alpha$
with $M \cap \ka \in A_\alpha$.
By the elementarity, we have that 
for every $\alpha_0,\dotsc,\alpha_n \in E$ and $\beta_0, \dotsc, \beta_m<\la$,
the family $\{x \in \pkl \mid x \in A_{\alpha_i}$ for every $i<n$ and 
$\beta_i \in x$ for every $i<m\}$ is non-empty.
Thus we can take a fine ultrafilter $U$ over $\pkl$
such that $U \notin S^{\calA^\alpha}$ for every $\alpha \in E$.
Hence $\{S^{\calA^\alpha} \mid \alpha \in E\}$ is not a cover.
\end{proof}

For $\alpha<\nu$,
since $S^{\calA^\alpha}$ is a closed $G_\delta$-subset of the compact Hausdorff space $\mathrm{Fine}(\pkl)$,
there is a continuous map $f_\alpha:\mathrm{Fine}(\pkl) \to [0,1]$
such that $S^{\calA^\alpha}=f_\alpha^{-1}\{0\}$.
Let $D=\{f_\alpha \mid \alpha<\nu\}$.
Let $g:\mathrm{Fine}(\pkl) \to \{0\}$ be the constant function.
We shall show that $g$ and $D$ witness $t(X_\delta) \ge \ka$.

First we check that $g \in \overline{D}^\delta$.
Take an open neighborhood $O$ of $g$ in $X_\delta$.
By the definition of the topology of $X_\delta$,
there are  $U_n \in \mathrm{Fine}(\pkl)$ $(n<\om)$
such that $\{h \in X \mid h(U_n)=0$ for $n<\om\} \subseteq O$.
By the assumption,
there is $\alpha<\nu$ such that
$U_n \in S^{\calA^\alpha}$ for every $n<\om$.
Then $f_\alpha(U_n)=0$ for $n<\om$,
hence $f_\alpha\in D \cap O$.

Finally we show that if a  set $E \subseteq \nu$ has cardinality $<\ka$,
then $g \notin \overline{ \{f_\alpha \mid \alpha \in E\}}^\delta$,
which means that $t(X_\delta) \ge \ka$.
Suppose to the contrary that $g \in \overline{ \{f_\alpha \mid \alpha \in E\}}^\delta$.
For each $U \in \mathrm{Fine}(\pkl)$,
the set $\{h \in C_p(X) \mid h(U)=0\}$ is an open neighborhood of $g$ in $X_\delta$.
So we can pick $\alpha \in E$ with $f_\alpha(U)=0$.
Since $S^{\calA^\alpha}=f_\alpha^{-1}\{0\}$,
we have $U \in S^{\calA^\alpha}$.
This shows that $\{S^{\calA^\alpha}\mid \alpha \in E\}$ is a cover of $\mathrm{Fine}(\pkl)$,
contradicting to Fact \ref{1.3}.
\end{proof}

\section{On the assumptions of Proposition \ref{1.2}}
Now let us discuss 
when the assumptions of Proposition \ref{1.2}
hold.

For a filter $F$ over the set $S$,
let $F^+=\{X \in \p(S) \mid S \setminus X \notin F\}$.
$F^+$ is the complement of the dual ideal of $F$. 
An element of $F^+$ is called an \emph{$F$-positive set}.
For $X \in F^+$, let $F \restriction X=\{Y \subseteq S \mid Y \cup (S \setminus X) \in F\}$.
$F \restriction X$ is  the filter over $S$ generated by $F \cup \{X\}$.

\begin{lemma}\label{1.14}
Let $S$ be an uncountable set,
and $\{U_n \mid n<\om\}$ a family of ultrafilters over $S$.
Let $F=\bigcap_{n<\om} U_n$.Then the following are equivalent:
\begin{enumerate}
\item There is a countable partition $\calA$ of $S$ such that
$A \notin U_n$ for every $A \in \calA$ and $n<\om$.
%\item There is a descending sequence $\seq{B_i \mid i<\om}$ of subsets of $S$
%such that $B_i \in F$ and $\bigcap_{i<\om} B_i =\emptyset$.
\item For every $X \in F^+$, the filter $F \restriction X$ is $\om_1$-incomplete.
\end{enumerate}
\end{lemma}
\begin{proof}
(1) $\Rightarrow$ (2) is clear.
For (2) $\Rightarrow$ (1),
we define $C_\alpha \subseteq S$ $(\alpha<\om_1)$ as follows:
First, let $C_0=S \in F$.
Suppose $C_\gamma$ is defined for every $\gamma <\alpha$ so that:
\begin{enumerate}
\item $\seq{C_\gamma \mid \gamma<\alpha}$ is a $\subseteq$-decreasing sequence of $F$-positive sets.
%\item $C_\gamma \in F^+$.
\item $C_\gamma=\bigcap_{\delta<\gamma} C_\delta$ if $\gamma$ is limit.
\item If $\gamma+1<\alpha$,
then there are $C_{\gamma,i} \in F\restriction C_\gamma$ for $i<\om$
such that $C_\gamma=C_{\gamma,0}\supseteq C_{\gamma,1} \supseteq\cdots$
and 
$C_{\gamma+1}=\bigcap_{i<\om} C_{\gamma,i} \notin F \restriction C_\gamma$.
\end{enumerate} 
Suppose $\alpha=\beta+1$.
Since $C_\beta \in F^+$ and $F \restriction C_\beta$ is $\om_1$-incomplete,
we can find 
$C_{\beta,i} \subseteq C_\beta$ for $i<\om$ 
such that $C_\beta=C_{\beta,0}\supseteq C_{\beta,1} \supseteq \cdots$,
$C_{\beta,i} \in F\restriction C_\beta$ for every $i<\om$,
and $\bigcap_{i<\om} C_{\beta,i} \notin F \restriction C_\beta$.
If $\bigcap_{i<\om} C_{\beta,i} \notin F^+$, then we finish this construction.
If $\bigcap_{i<\om} C_{\beta,i} \in F^+$, then let $C_\alpha=\bigcap_{i<\om} C_{\beta,i}$.

If $\alpha$ is limit and $\bigcap_{\gamma<\alpha} C_\gamma \notin F^+$,
then we finish the construction,
otherwise let $C_\alpha=\bigcap_{\gamma<\alpha} C_\gamma$.
 
We claim that this construction have to be finished at some $\gamma<\om_1$.
If not, then let $D_\alpha=C_\alpha \setminus C_{\alpha+1}$ for $\alpha<\om_1$.
Since $C_{\alpha+1} \notin F \restriction C_\alpha$, we have $D_\alpha \in F^+$.
For $\alpha<\om_1$, there is some $n_\alpha<\om$ with $D_\alpha \in U_{n_\alpha}$.
Hence there is some $n<\om$ such that
the set $\{\alpha<\om_1 \mid n_\alpha =n\}$ is uncountable.
Pick $\alpha<\beta<\om_1$ with $n=n_\alpha=n_\beta$.
We have $D_\alpha, D_\beta \in U_n$, hence $D_\alpha \cap D_\beta \neq \emptyset$.
This is impossible.

Now suppose $\{C_\alpha \mid \alpha<\gamma \}$ is defined as above but
$C_\gamma$ cannot be defined.
If $\gamma$ is limit, we have $\bigcap_{\alpha<\gamma} C_\alpha \notin F^+$.
%Since $\bigcap_{\alpha<\gamma} C_\alpha \notin F^+$, 
By shrinking each $C_\alpha$,
we may assume that $\bigcap_{\alpha<\gamma} C_\alpha =\emptyset$.
By the construction,
for $\alpha<\gamma$, there is a sequence $\seq{C_{\alpha,i} \mid i<\om}$ as above.
Let $A_{\alpha,i}=C_{\alpha,i} \setminus C_{\alpha,i+1}$.
We check that $A_{\alpha,i} \notin F^+$; 
Since $C_{\alpha,i}, C_{\alpha,i+1} \in F\restriction C_\alpha$,
we have that $A_{\alpha,i} =C_{\alpha,i} \setminus C_{\alpha,i+1} \notin (F \restriction C_\alpha)^+$.
Furthermore, since $A_{\alpha,i} \subseteq C_{\alpha,i} \subseteq C_\alpha$,
we have $A_{\alpha,i} \notin F^+$.
Now the family $\{A_{\alpha,i} \mid \alpha<\gamma, i<\om\}$ is a countable partition of $S$
such that $A_{\alpha,i} \notin F^+$, so $A_{\alpha,i} \notin U_n$ for every $n<\om$.
Thus (1) holds.

If $\beta$ is successor, say $\gamma=\beta+1$,
then there are $C_{\beta,i} \in (F \restriction C_\beta)^*$ ($i<\om$)
such that $C_\beta \supseteq C_{\beta,0} \supseteq C_{\beta,1} \supseteq \cdots$
and $\bigcap_{i<\om} C_{\beta,i} \notin F^+$.
As in the limit case, we may assume that $\bigcap_{i<\om} C_{\beta,i}=\emptyset$.
For each $\alpha \le \beta$, let $A_{\alpha,i} = C_{\alpha,i} \setminus C_{\alpha,i+1}$.
Then $\{A_{\alpha,i} \mid i<\om, \alpha \le \beta\}$ is a required partition.
\end{proof}

We will use the generic ultrapower argument.
See Foreman \cite{Foreman} for the generic ultrapower.
Here we present some basic definitions and facts.

Let $F$ be a  filter over the set $S$.
For a cardinal $\ka$, we say that $F$ is \emph{$\ka$-saturated} if for every family $\{X_\alpha\mid \alpha<\ka\}$
of $F$-positive sets, there are $\alpha<\beta<\ka$ with $X_\alpha \cap X_\beta \in F^+$.
For $F$-positive sets $X$ and $Y$, define $X \le_F Y$ if $X \setminus Y \notin F^+$.
Let $\bbP_F$ be the poset $F^+$ with the order $\le_F$.
Note that 
for $X, Y \in \bbP_F$, $X$ is compatible with $Y$ in $\bbP_F$ if and only if
$X \cap Y \in F^+$, and 
$\bbP_F$ has the $\ka$-c.c. if and only if $F$ is $\ka$-saturated.

If $G$ is a  $(V, \bbP_F)$-generic filter,
then $G$ is a $V$-ultrafilter over $S$,
that is the following hold:
\begin{itemize}
\item $S \in G$, $\emptyset \notin G$.
\item $X \cap Y \in G$ for every $X, Y \in G$.
\item For $X, Y \in V$, if $X \in G$ and $X \subseteq Y \subseteq S$ then $Y \in G$.
\item For every $X \subseteq S$ with $X \in V$,
either $X \in G$ or $S \setminus X \in G$.
\end{itemize}
Hence we can take 
the \emph{generic ultrapower} of $V$ by $G$.
For a $(V, \bbP_F)$-generic filter $G$,
let $\mathrm{Ult}(V, G)$ be the generic ultrapower of $V$ by $G$,
and $j:V \to \mathrm{Ult}(V,G)$ be the elementary embedding induced by $G$.
If $\mathrm{Ult}(V,G)$ is well-founded,
we identify $\mathrm{Ult}(V,G)$ with its transitive collapse.
We say that $F$ is \emph{precipitous} if for every $(V,\bbP_F)$-generic $G$,
$\mathrm{Ult}(V,G)$ is well-founded.

%For a filter $F$ over $S$, we say that $F$ has the \emph{disjointing property}
%if for every antichain $\calA$ in $\bbP_F$,
%there is $\{C_A \mid A \in \calA\}$ such that
%$C_A \in F$, and $(A \cap C_A) \cap (A' \cap C_{A'})=\emptyset$ for every
%distinct $A, A' \in \calA$.
\begin{fact}\label{1.4}
Let $\ka$ be an uncountable cardinal,
and $F$ a $\ka$-complete filter over $S$.
If $F$ is $\ka^+$-saturated,
then $F$ is precipitous.
%Moreover in the generic extension of $V$ via $\bbP_F$,
%the generic ultrapower of $V$ is closed under $\ka$-sequences.
\end{fact}
%
%
%%For a cardinal $\mu$,
%%define $\beth_\alpha(\mu)$ as follows:
%%$\beth_0(\mu)=\mu$,
%%$\beth_{\alpha+1}=2^{\beth_\alpha(\mu)}$,
%%and $\beth_{\alpha}(\mu)=\sup_{\beta<\alpha} \beth_\beta(\alpha)$ is $\alpha$ if limit.
%
%%For an uncountable cardinal $\ka$ and a cardinal $\la \ge \ka$,
%%we say that $S \subseteq \pkl$ is \emph{stationary in $\pkl$}
%%if for every function $f:[\la]^{<\om} \to \la$,
%%there is $x \in S$ such that $f``[x]^{<\om} \subseteq x$.
%%For a function $f:[\la]^{<\om} \to \la$,
%%let $C_f=\{x \in \pkl \mid f``[x]^{<\om} \subseteq x\}$.
%%Let $\mathcal C_{\ka,\la}$ be the family of subsets of $\pkl$
%%such that $X \in 
%%\mathcal C_{\ka,\la} \iff C_f \subseteq X$ for some 
%%$f :[\la]^{<\om} \to \la$.
%%$\calC_{\ka,\la}$ forms an $\om_1$-complete fine filter over $\pkl$.
%%Note that $X \subseteq \pkl$ is stationary in $\pkl$ if and only if $X \in \calC_{\ka,\la}^+$.

\begin{prop}\label{1.5}
Let $\ka$ be an uncountable cardinal,
and $\la>2^\ka$.
Let $\{U_n \mid n<\om\}$ be a family of $\om_1$-incomplete fine ultrafilters over $\pkl$.
If the filter $F=\bigcap_{n<\om} U_n$ is $\om_1$-complete,
then there is a weakly Mahlo cardinal $<2^\om$,
and $\ka>(2^\om)^+$.
\end{prop}
\begin{proof}
First note that for $X \subseteq \pkl$,
$X \in \bbP_F$ if and only if $X \in U_n$ for some $n<\om$.
For each $X \in \bbP_F$,
let $I_X=\{n<\om \mid X \in U_n\} \neq \emptyset$.

We shall prove a series of claims.

\begin{claim}\label{claim1.6}
For $X, Y \in \bbP_F$,
$X \le_F Y \iff I_X \subseteq I_Y$,
and $X$ is compatible with $Y \iff I_X \cap I_Y \neq \emptyset$.
\end{claim}
\begin{proof}
If $I_X \nsubseteq I_Y$, pick $n \in I_X \setminus I_Y$.
We know $X \in U_n$ but $Y \notin U_n$,
hence $X \setminus Y \in U_n$, and $X \not \le_F Y$.
For the converse, suppose $X \not \le_F Y$.
Then $X \setminus Y \in F^+$, and there is $n<\om$ with
$X \setminus Y \in U_n$. We have $X \in U_n$ but $Y \notin U_n$,
so $n \in I_X \setminus I_Y$ and $I_X \nsubseteq I_Y$.

If $X$ is compatible with $Y$, then
there is $Z \le_F X, Y$.
We may assume $Z \subseteq X \cap Y$,
then $I_Z \subseteq I_X \cap I_Y$, so $I_X \cap I_Y \neq \emptyset$.
For the converse,
if $I_X \cap I_Y \neq \emptyset$,
take $n \in I_X \cap I_Y$.
Then $X, Y \in U_n$, so $X \cap Y \in U_n$.
Hence $X \cap Y \in F^+$, and
we have $X \cap Y \le_F X, Y$.
\end{proof}

%\begin{claim}
%For every $X \in \bbP_F$,
%there are $Y, Z \le_F X$ such that $Y$ is incompatible with $Z$.
%In particular, the forcing with $\bbP_F$ must add new  subset.
%\end{claim}
%\begin{proof}
%First note that $I_X $ has at least two element;
%If $I_X=\{n\}$, then  we have that $F \restriction X=U_n$.
%But $U_n$ is not $\om_1$-complete.
%So we can take distinct $n, m \in I_X$.
%Take $X' \in U_n$ with $X' \notin U_m$.
%Let $Y=X \cap X'$ and $Z=X \setminus X'$,
%$Y \in U_n$ and $Z \in U_m$, hence $Y,Z \in \bbP_F$ and $Y$ is incompatible with $Z$.
%\end{proof}

\begin{claim}\label{1.8}
$\bbP_F$ has the c.c.c.,
and has a dense subset of size $\le 2^\om$.
\end{claim}
\begin{proof}
Take an uncountable family $\{X_\alpha \mid \alpha<\om_1\} \subseteq \bbP_F$.
Since $I_{X_\alpha} \subseteq \om$,
there must be $\alpha<\beta<\om_1$ with $I_{X_\alpha} \cap I_{X_\beta} \neq \emptyset$.
Hence $X_{\alpha}$ is compatible with $X_{\beta}$ by Claim \ref{claim1.6}.

For $X, Y \in \bbP_F$,
define $X \approx Y$ if $X \le_F Y$ and $Y \le_F X$.
By Claim \ref{claim1.6},
$X \approx Y$ if and only if $I_X=I_Y$.
Hence there are at most $2^\om$ many equivalence classes,
and we can take a dense subset in $\bbP_F$ of size $\le 2^\om$.
\end{proof}
From now on, we identify $\bbP_F$ with its dense subset of size $\le 2^\om$.

Now, we know that $F$ is $\om_1$-complete and $\om_1$-saturated.
Hence $F$ is precipitous by Fact \ref{1.4}.
Take a $(V, \bbP_F)$-generic $G$,
and let $j:V \to \mathrm{Ult}(V,G)$ be the generic elementary embedding induced by $G$.
Since $F$ is precipitous, we can identify $\mathrm{Ult}(V,G)$ with its transitive collapse $M$.
For a map $f :\pkl \to V$ with $f \in V$,
let $[f]$ be the equivalence class of $f$ by $G$.
If $id:\pkl \to \pkl$ is the identity map,
then we have $j``\lambda \subseteq [id]$ because $F$ is a fine filter.
Moreover $\la \le \size{[id]}^M < j(\ka)$.

Let $\mu$ be the critical point of $j$. $\mu$ is regular uncountable in $V$.
Since $\bbP_F$ has the c.c.c., $\mu$ remains regular in $V[G]$, and so does in $M$.
%We prove that $\mu$ is as required.

\begin{claim}\label{1.9}
$\mu$ is weakly Mahlo in $V$.
\end{claim}
\begin{proof}
Let $C \subseteq \mu$ be a club in $\mu$ with $C \in V$.
Then $\mu \in j(C)$, hence
it holds that the statement  ``$j(C)$ contains a regular cardinal'' in $M$.
By the elementarity of $j$,
$C$ contains a regular cardinal in $V$.
\end{proof}
Note that $\mu$ is in fact weakly $\mu$-Mahlo.

\begin{claim}
For every cardinal $\nu$,
we have 
$(2^\nu)^V=(2^\nu)^{V[G]}$.
\end{claim}
\begin{proof}
Since $\bbP_F$ has a dense subset of size $(2^\om)^V$ and
has the c.c.c., there are at most $(2^\om)^\om$-many antichains in $\bbP_F$,
and at most $((2^\om)^\nu)^V=(2^\nu)^V$ many canonical names for subsets of $\nu$.
Hence we have $(2^\nu)^V=(2^\nu)^{V[G]}$.
\end{proof}

Thus for each cardinal $\nu$,
we can let $2^\nu$ denote $(2^\nu)^V$ and $(2^\nu)^{V[G]}$.
%In particular we have $\beth_\alpha(\nu)^V=\beth_\alpha(\nu)^{V[G]}$
%for every cardinal $\nu$ and ordinal $\alpha$,
%we can let $\beth_\alpha(\nu)$ denote it.
In addition, since $\bbP_F$ has the c.c.c.,
we have $(\nu^+)^V=(\nu^+)^{V[G]}$ for every cardinal $\nu$,
and we can let $\nu^+$ denote $(\nu^+)^V$ and $(\nu^+)^{V[G]}$.
Note that $(\nu^+)^M  \le \nu^+$ for every cardinal $\nu$.

%Since $\mu$ is the critical point of $j$,
%we have that $A=j(A) \cap \mu \in M$ for every $A \subseteq \mu$ with $A \in V$.
%Hence $(\mu^+)^V \le (\mu^+)^M$.
%Because $\bbP_F$ has the c.c.c., we have $(\mu^+)^V=(\mu^+)^{V[G]} \ge (\mu^+)^M$.
%This means that $\mu^+<j(\mu)$.

\begin{claim}\label{1.11}
$j(2^\om)<(2^\om)^+=j((2^\om)^+)<\ka$.
\end{claim}
\begin{proof}
If $j(2^\om) \ge (2^\om)^+$,
then $M$ has at least $(2^\om)^+$ many subsets of $\om$,
so we have $(2^\om)^{V[G]} \ge (2^\om)^+$, this contradicts to the previous claim,
and we have $j(2^\om)<(2^\om)^+$.
Thus $(2^\om)^+ \le j((2^\om)^+)=(j(2^\om)^+)^M \le 
j(2^\om)^+
\le (2^\om)^+$,
so we have $j((2^\om)^+)=(2^\om)^+$.

Finally, since $(2^\om)^+ \le (2^\ka)^+ \le \la$, we have $j((2^\om)^+)=(2^\om)^+ \le \la<j(\ka)$.
Then $(2^\om)^+<\ka$ by the elementarity of $j$.
\end{proof}

We completes the proof by showing the following:
\begin{claim}\label{1.12}
$\mu<2^\om$.
\end{claim}
\begin{proof}
Note that $j(\mu)>\mu^+$
since $\p(\mu)^V \subseteq \p(\mu)^M$ 
and $j(\mu)$ is a limit cardinal in $M$.

Since $j(2^\om)<(2^\om)^+$, we have
$\mu \neq 2^\om$. Next we show $2^\om>\mu$.
If not, then $2^\om<\mu$.
Take $X \in G$ such that
$X \Vdash_{\bbP_F}$``the critical point of $j$ is $\mu$''.
Then the filter $F \restriction X$ is in fact $\mu$-complete.
We use the following well-known fact by Tarski:
\begin{fact}[Tarski, e.g. see Kanamori \cite{Ka}]
Let $S$ be an uncountable set,
and $F$ a filter over $S$.
If $F$ is $(2^\om)^+$-complete and $\om_1$-saturated,
then there is $Y \in F^+$ such that $F \restriction Y$ is an ultrafilter.
\end{fact}

Since $2^\om<\mu$,
by Tarski's theorem there is $Y \in (F \restriction X)^+$ such that
$F \restriction Y$ is an ultrafilter.
Hence $F \restriction Y =U_n$ for some $n$,
and $U_n$ is $\mu$-complete. This is a contradiction.
\end{proof}
\end{proof}

\begin{lemma}\label{normal}
Let $X$ be a countably tight $T_1$ space.
Then there is a normal countably tight $T_1$ space $Y$ with
$t(Y_\delta)=t(X_\delta)$.
\end{lemma}
\begin{proof}
Fix a point $p^* \in X$ such that
there is $A \subseteq X$ with $p^* \in \overline{A}^\delta$,
but  no $B \subseteq A$ with $\size{B}<t(X_\delta)$ and $p \in \overline{B}^\delta$.
Let $Y$ be the space $X$ equipped with the following topology:
\begin{enumerate}
\item Every $q \in X \setminus\{p^*\}$ is isolated in $Y$.
\item A local base for $p^*$ in $Y$ is the same to in $X$.
\end{enumerate}
It is easy to check that
$Y$ is a normal countably tight $T_1$ space with $t(Y_\delta)=t(X_\delta)$.
\end{proof}

%\begin{claim}
%$\ka \ge (2^\gamma)^{+\gamma}$.
%\end{claim}
%\begin{proof}
%First we check that $\ka >2^\gamma$.
%If $2^\gamma \ge \ka$,
%then $j(2^\gamma) \ge j(\ka)>\la$.
%$j(2^\gamma)=\size{\p(j(\gamma))^M}^M>\la$,
%so we can take a surjection from $\p(j(\gamma))$ onto $\la$.
%However, because $j(\gamma)<\ka$,
%we have $\size{\p(j(\gamma))} \le 2^\ka<\la$, this is a contradiction.
%
%To show that $(2^\gamma)^{+\gamma} \le \ka$,
%it is enough to see that
%$(2^\gamma)^{+\alpha}<\ka$ for every $\alpha<\gamma$.
%Let $\delta=2^\gamma$. As before, since $j(\gamma)<\ka$,
%we know $j(\delta)\le (2^\ka)^+<\la$.
%Thus $j(\delta^{+\alpha})=(j(\delta)^{+\alpha})^M \le j(\delta)^{+\alpha}
%\le ((2^\ka)^+)^{+\alpha} <\la<j(\ka)$,
%and $\delta^{+\alpha}<\ka$.
%\end{proof}

Now we have the theorems.

\begin{cor}
There is a normal countably tight $T_1$ space $X$ 
such that $t(X_\delta) >2^\om$.
\end{cor}
\begin{proof}
Let $\ka=(2^\om)^+$.
$\ka$ is not $\om_1$-strongly compact,
and so there is $\la \ge \ka$
such that $\pkl$ has no $\om_1$-complete fine ultrafilter.
$\la$ can be arbitrary large, so we may assume $\la>2^\ka$.
By Proposition \ref{1.2} and Lemmas \ref{1.14}, \ref{normal},
it is enough to show that for every countable family
$\{U_n \mid n<\om\}$ of fine ultrafilters over $\pkl$
and $X \in F^+=(\bigcap_{n<\om} U_n)^+$, the filter $F \restriction X$ is $\om_1$-incomplete.
Let $I=\{n<\om \mid X \in U_n\}$.
Then it is easy to check that
$F \restriction X=\bigcap_{n \in I} U_n$.
Because $\ka=(2^\om)^+$, 
we know that $\bigcap_{n \in I} U_n$ is $\om_1$-incomplete by Proposition \ref{1.5}.
\end{proof}

\begin{cor}
Suppose there is no weakly Mahlo cardinal $< 2^\om$.
If there is no $\om_1$-strongly compact cardinal below a cardinal $\ka$,
then there is a normal countably tight $T_1$ space $X$ with $t(X_\delta) \ge \ka$.
\end{cor}
\begin{proof}
We know that $\ka$ is not $\om_1$-strongly compact, so there is a large $\la>\ka$
such that $\pkl$ cannot carry an $\om_1$-complete fine ultrafilter.
By the assumption and Proposition \ref{1.5},
there is no countable family of fine ultrafilters $\{U_n\mid n<\om\}$ over $\pkl$
with $\bigcap_{n<\om} U_n$ $\om_1$-complete.
Again, by 
Proposition \ref{1.2} and Lemmas \ref{1.14}, \ref{normal}
we can take a normal countably tight $T_1$ space $X$ with $t(X_\delta) \ge \ka$.
\end{proof}

\begin{question}
Is the equality
that ``the least $\om_1$-strongly compact$=\sup\{t(X_\delta) \mid X$ is normal countably tight $T_1\}$''
provable from ZFC without any assumptions?
\end{question}

%\begin{question}
%Let $S$ be an uncountable set,
%and $U_n$ ($n<\om$) be $\om_1$-incomplete ultrafilters over $S$.
%Is it possible that the filter $\bigcap_n U_n$ is $\om_1$-complete?
%\end{question}

\begin{question}
In the theorems, can we replace ``countably tight'' by ``Fr\'echet-Urysohn''?
For instance, is there a ZFC-example of a Fr\'echet-Urysohn space $X$ with
$t(X_\delta)>2^\om$?
\end{question}
Note that our space $C_p(\mathrm{Fine}(\pkl))$ is not Fr\'echet-Urysohn;
It is known that for a compact Hausdorff space $Y$,
$C_p(Y)$ is Fr\'echet-Urysohn if and only if $Y$ is scattered
(Pytkeev \cite{P2}, Gerlits \cite{Gerlits}).
However the space $\mathrm{Fine}(\pkl)$ is not scattered.

%\section{atodekesu}
%
%\begin{lemma}
%Let $\bbP$ be a $\sigma$-centered poset.
%Then $\bbP$ adds an unbounded real.
%\end{lemma}
%Let $\bbP=\bigcup_n A_n$ where $A_n$ is a centered set.
%
%Let $\dot r$ be a name such that
%$p \Vdash$``$\dot r:\om \to 2, \dot r \notin V$''.
%For each $q \le p$,
%let $T_q=\{t \in {}^{<\om }2 \mid \exists r \le q r \Vdash t=\dot r \restriction \dom(t)\}$.
%It is easy to check that $T_q$ is a downward closed, has no end-point.
%Let $B_q=\{f \in {}^\om 2\mid f \restriction n \in T_q$ for every $n<\om\}$.
%$B_q$ is a closed subset of Cantor space ${}^\om 2$.
%In addition if $q' \le q$ then $B_{q'} \subseteq B_q$.
%$A_n$ is centered, so $\{B_q \mid q \in A_n\}$ has the finite intersection property,
%and we can find $r_n \in \bigcap_{q \in A_n} B_q$.
%Then define $\dot f$ by $p \Vdash$``$\dot f(n)=m \iff m<\om$ is the least $m$ with
%$\dot r \restriction m\neq r_n \restriction m$''.
%We see that $p \Vdash \dot f \not \le g$
%for every $g \in {}^\om \om \cap V$.
%If not, then there is $q \le p$ and $g \in {}^\om \om \cap V$
%such that
%$q \Vdash \dot f \le g$.
%Pick $n$ with $q \in A_n$.
%Since $r_n \in B_q$,
%there is $r \le q$ such that
%$r \Vdash$``$\dot r \restriction g(m)=r_n \restriction g(m)$.
%Then $r \Vdash \dot r(n) >g(m)$,
%a contradiction.

\subsection*{Acknowledgments}
This research was supported by 
JSPS KAKENHI Grant Nos. 18K03403 and 18K03404.

\printindex

\end{document}